\documentclass[letterpaper, 10 pt, conference]{ieeeconf}  

\IEEEoverridecommandlockouts                              

\usepackage{graphicx,amsmath}
\usepackage{amssymb}
\usepackage{amsfonts}
\usepackage{calrsfs}
\usepackage{mathrsfs}
\usepackage{caption}
\usepackage{subcaption}
\usepackage{multirow}
\usepackage{xcolor}
\usepackage{graphics}
\usepackage{float}
\usepackage{graphicx}
\usepackage{calrsfs}
\usepackage{mathtools}
\usepackage{amsmath}
\usepackage{psfrag}
\usepackage{blindtext}
\newtheorem{lemma}{Lemma}
\newtheorem{theorem}{Theorem}

\newtheorem{problem}{Problem}

\makeatletter
\renewcommand*\env@matrix[1][*\c@MaxMatrixCols c]{%
	\hskip -\arraycolsep
	\let\@ifnextchar\new@ifnextchar
	\array{#1}}
	\makeatother
\mathtoolsset{showonlyrefs}
 \usepackage{graphicx}
\usepackage{amssymb}
\usepackage{epstopdf}
\usepackage{xcolor,colortbl}
\usepackage{hyperref}
\usepackage{pbox}
\usepackage{epsfig,psfrag}
\usepackage{color}
\usepackage{amsmath}
\usepackage{mathrsfs}  
\usepackage{amsbsy}
\usepackage{lipsum} 
\usepackage{pgf}
\usepackage{tikz}
\usetikzlibrary{arrows, automata, shapes.multipart, chains, positioning, fit, graphs, spy, shapes, backgrounds, patterns,decorations.pathmorphing}
\usepackage{verbatim}
\tikzstyle{block} = [draw, fill=blue!20, rectangle, 
minimum height=1 cm, minimum width=2cm, rounded corners=0.1cm]
\tikzstyle{block2} = [draw, fill=black!15, rectangle, 
minimum height=0.75 cm, minimum width=1.5cm, rounded corners=0.1cm]
\tikzstyle{faultyblock} = [draw, fill=red!20, rectangle, 
minimum height=1 cm, minimum width=2cm, rounded corners=0.1cm]
\tikzstyle{reconblock} = [draw, fill=green!20, rectangle, 
minimum height=1 cm, minimum width=2cm, rounded corners=0.1cm]
\tikzstyle{sum} = [draw, fill=black!15, circle, minimum width=0.5cm]
\tikzstyle{input} = [coordinate]
\tikzstyle{output} = [coordinate]
\tikzstyle{pinstyle} = [pin edge={to-,thin,black}]
\usepackage{float}
\usepackage{isomath}
\usepackage{subcaption}

\title{\LARGE \bf
Unit-Vector Control Design under Saturating Actuators*
}

\author{Andevaldo da Encarnação Vitório$^1$, Pedro Henrique Silva Coutinho$^2$, Iury Bessa$^{3}$, \\ Victor Hugo Pereira Rodrigues$^2$, Tiago Roux Oliveira$^2$
\thanks{*This work was supported by the Brazilian agencies CNPq (Grant numbers: 407885/2023-4 and 309008/2022-0), CAPES, FAPEAM, and FAPERJ.} 
\thanks{$^{1}$ Andevaldo Vitório is with the Graduate Program in Electrical Engineering, Federal University of Amazonas, Manaus, Amazonas, Brazil.
			        {\tt\small andevaldo.vitorio@ufam.edu.br} }   
\thanks{$^{2}$ Pedro H. S. Coutinho, Victor H. P. Rodrigues, and Tiago Roux Oliveira are with the Department of Electronics and Telecommunication
Engineering, State University of Rio de Janeiro, Brazil.
			        {\tt\small phcoutinho@eng.uerj.br, victor.rodrigues@uerj.br, tiagoroux@uerj.br}} 
\thanks{$^{3}$ Iury Bessa is with the Department of Electricity, Federal University of Amazonas, Manaus, Amazonas, Brazil.
			        {\tt\small iurybessa@ufam.edu.br} }
}

\begin{document}

\maketitle
\thispagestyle{empty}
\pagestyle{empty}

\begin{abstract}
This paper deals with unit vector control design for multivariable polytopic uncertain systems under saturating actuators. For that purpose, we propose LMI-based conditions to design the unit vector control gain such that the origin of the closed-loop system is finite-time stable. Moreover, an optimization problem is provided to obtain an enlarged estimate of the region of attraction of the equilibrium point for the closed-loop system, where the convergence of trajectories is ensured even in the presence of saturation functions. Numerical simulations illustrate the effectiveness of the proposed approach.
\end{abstract}

\section{Introduction}
\label{sec:intro}

Sliding mode control is recognized by its inherent robustness properties, strong convergence properties, and the simplicity of its algorithmic implementation. In particular, the robustness of sliding mode controllers has been explored to deal with different phenomena, such as time-delays~\cite{chen2024sliding}, model uncertainties~\cite{Choi2007}, and the presence of additive disturbances~\cite{Cui2019} or faults~\cite{Wang2023}.
However, the literature still lacks systematic design methodologies, like those based on convex optimization that are already traditional for robust and nonlinear control of polytopic differential inclusions. 

In this regard, several recent studies have investigated the use of particular Lyapunov function structures to enable the design of sliding mode controllers based on semidefinite programming with linear matrix inequalities (LMIs) constraints. For example, LMI-based conditions have been successfully applied to design sliding-mode controllers for systems with disturbances in~\cite{Roy2020}. In~\cite{chen2024sliding}, LMI-based stabilization conditions are presented for impulsive sliding-mode control of systems subject to time delays and input disturbances. Recently,~\cite{geromel2024lmi,geromel2024multivariable} investigated the design of super-twisting control algorithms for polytopic uncertain multivariable systems. However, none of the aforementioned papers explicitly deal with the presence of saturating actuators.

The presence of saturating actuators is also a relevant issue that is frequently ignored. However, it is well known that the saturation might deteriorate the closed-loop performance, and lead to the loss of stability guarantees since the saturation modifies the domain of attraction. Therefore, sliding-mode controllers have also been applied to systems with saturating inputs~\cite{Corradini2007,Han2020}. However, there are still a few LMI-based design conditions for sliding-mode controllers for systems with saturating actuators. For linear time-invariant models, a unit vector control (UVC) approach designed based on LMIs is proposed for dealing with saturated inputs in~\cite{Zaafouri2017}. However, that research is not able to deal with uncertain systems. Moreover, the region of admissible initial conditions associated with a guaranteed finite time is not provided.

In this regard, this paper investigates the problem of UVC for polytopic uncertain systems with saturating inputs. Although the UVC can be viewed
as a saturated-by-construction control law, the control signal bound will depend on the control gain employed. Thus, two main aspects should be addressed for the proper UVC design under saturating actuator: \textit{(i)} How to design the control gain such that the convergence is ensured even in the presence of saturation? \textit{(ii)} What is the set of initial conditions for which the convergence is ensured (in the presence of saturation) for a guaranteed reaching time? In this paper, we provide novel finite-time stabilization conditions based on LMI constraints that allow to design UVCs that solve the above-mentioned problems. Moreover, we propose an optimization problem to estimate the region of attraction for which the convergence is guaranteed for the given reaching time bound, even in the presence of saturation.  In short, this paper presents the following original contributions: 
\begin{itemize}
    \item Novel LMI-based design conditions for polytopic uncertain systems under saturating actuators;
    \item The proposed design conditions guarantee the finite-time stabilization for those systems;
    \item We provide an optimization problem that allows for obtaining an enlarged estimate of the domain of attraction for the origin of those systems associated with a given finite time bound.
\end{itemize}

\textbf{Outline.} This paper is organized as follows. Section~\ref{sec:preliminaries} formulates the problem of unit vector control of polytopic uncertain systems under actuator saturation. Section~\ref{sec:main} presents the main results, including the LMI-based design conditions. Section~\ref{sec:sim} presents numerical simulations that demonstrate the efficacy of the proposed approach. Finally, Section~\ref{sec:conclusion} draws the main conclusions.

\textbf{Notation.}
$\mathbb{R}^{n}$ and $\mathbb{R}_{\geq 0}$ denote respectively the Euclidean space of $n$-dimensional real numbers and non-negative real numbers. The set of positive integers is $\mathbb{N}$, and the set $\mathbb{N}_{\leq p} = \{1,\ldots,p\}$, for some $p>1$. For a matrix $X$, $X \succ (\prec)~0$ means that $X$ is a positive (negative) definite matrix; $X^\top$ is its transpose; the identity matrix of dimension $n$ is denoted by $I_n$ and the null matrix of order $n\times m$ by $0_{n \times m}$; $\mathrm{diag}\{d_{1},\dots,d_{n}\}$ is a diagonal matrix with the elements/blocks $d_{1},\dots,d_{n}$ in the main diagonal. For a vector $v \in \mathbb{R}^{n}$, ${v}_{(\ell)}$ denotes its $\ell$-th entry.

\section{Problem Formulation}
\label{sec:preliminaries}

Consider the UVC system shown in Fig.~\ref{fig:diagram}.
\begin{figure}[!ht]
    \centering
    \includegraphics[width=\linewidth]{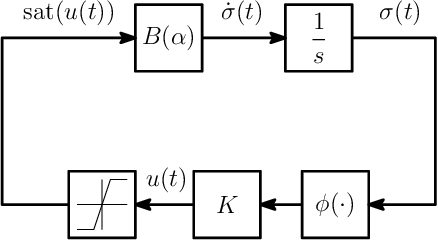}
    \caption{Block diagram of the UVC system.}
    \label{fig:diagram}
\end{figure}

Based on Fig.~\ref{fig:diagram}, we have that the plant is given by the following class of uncertain systems
\begin{align}\label{eq:plant}
  \dot{\sigma}  = B\mathrm{sat}(u),
\end{align}
where $\sigma \in \mathbb{R}^n$ is the state vector defined as $\sigma = (\sigma_1, \ldots, \sigma_n)$,
$u \in \mathbb{R}^m$ is the input vector.
Moreover, $B$ is constant but unknown, taking values in the set $\mathcal{B} = \mathrm{co}\{B_i\}_{i \in \mathbb{N}_{\leq N}}$.
This implies that it is possible to write
\begin{align}
  B = \sum_{i = 1}^N \alpha_i B_i,
\end{align}
where $\alpha = (\alpha_1,\ldots,\alpha_N)$
is the vector of uncertain parameters lying in the unit simplex
\begin{align}
  \Lambda = \left\lbrace\alpha \in \mathbb{R}^N: \sum_{i=1}^N \alpha_i = 1, \; \alpha_i \geq 0\right\rbrace.
\end{align}

Consider now the following unit vector control law:
\begin{align}\label{eq:control-uvc}
  u = K \phi(\sigma) =  K \frac{\sigma}{\|\sigma\|}
\end{align}
and let $\psi(u) = u - \mathrm{sat}(u)$ be the dead-zone nonlinearity.
Under this definition, the following closed-loop system can be obtained by substituting~\eqref{eq:control-uvc} in~\eqref{eq:plant}:
\begin{align}\label{eq:closed-loop-uvc}
  \dot{\sigma} = BK \frac{\sigma}{\|\sigma\|} - B\psi(u).
\end{align}
Thus, it is clear that the conditions for the design of the control gain change when we consider $u=K \sigma/||\sigma||$ or $u=\mathrm{sat}(K \sigma/||\sigma||)$.

Let $z = r(\sigma) \sigma$, where $r(\sigma) = 1/\sqrt{\|\sigma\|}$. In $z$-coordinates, the closed-loop system~\eqref{eq:closed-loop-uvc}
can be equivalently rewritten as
\begin{align}\label{eq:closed-loop}
  \dot{z} & = -\frac{1}{2}r(\sigma) \Pi_\sigma BK r(\sigma) z \\ & + r(\sigma) BK r(\sigma) z - r(\sigma) B \psi(u),
\end{align}
where $\Pi_\sigma = \sigma \sigma^\top/\|\sigma\|^2$ is a projection matrix which satisfies the
following properties: $\Pi_\sigma = \Pi_\sigma^\top$, $\Pi_\sigma^2 = \Pi_\sigma$, and $\|\Pi_\sigma\| = 1$, $\forall \sigma \in \mathbb{R}^n$~\cite{geromel2024multivariable}.


The problem addressed in this section is stated as follows.
\begin{problem}\label{problem:2}
Consider the uncertain system~\eqref{eq:plant} and the UVC law~\eqref{eq:control-uvc}.
Design a robust control gain $K$ such that the origin of the closed-loop system~\eqref{eq:closed-loop} is finite-time stable in the presence of saturating actuators.
\end{problem}

\section{Main Results}
\label{sec:main}

\subsection{Preliminaries Results}

\begin{lemma} [Adapted from~\cite{Tarbouriech2011}] \label{lemma:deadzone}
  Let the control input be defined as $u=Kz$ for all $z \in \mathbb{R}^{n}$ and a given $K \in \mathbb{R}^{m \times n}$, and the set $\mathcal{D}_{u}$ be  defined as
  \begin{equation}
    \label{eq:Du}
    \mathcal{D}_{u} = \left\lbrace z \in \mathbb{R}^{n} :
    |\left( K_{(\ell)} - L_{(\ell)} \right) z | \leq \bar u_{\ell}, \;
    \ell \in \mathbb{N}_{\leq m}\right\rbrace,
  \end{equation}
  \noindent for any matrix $L \in \mathbb{R}^{m \times n}$.
  If $z \in \mathcal{D}_{u}$, then
  \begin{equation}
    \label{eq:nl_sector}
    \psi(u)^{\top} U \left( \psi(u) - Lz \right) \leq 0,
  \end{equation}
  \noindent holds for any diagonal $U > 0$.
\end{lemma}
\begin{proof}
     If $z \in \mathcal{D}_{u}$, the following expression is derived:
\begin{align} \label{eq:saturation_interval}
    -\overline{u}_{(\ell)} \leq ( K - L)_{(\ell)} z \leq \overline{u}_{(\ell)}
\end{align}

To show that \eqref{eq:nl_sector} holds, three cases will be analyzed:
\begin{itemize}[leftmargin=*,nosep]
    \item \textbf{Case 1}: Control signal exceeds the maximum allowed limit, $u_{(\ell)} > \overline{u}_{(\ell)}$. In this case,
    \begin{align*}
        \psi(u_{(\ell)}) = u_{(\ell)} - \overline{u}_{(\ell)} = (K z)_{(\ell)} - \overline{u}_{(\ell)} > 0
    \end{align*}
    Thus, from \eqref{eq:Du}, it follows that
    \begin{align*}
        (K z)_{(\ell)} - \overline{u}_{(\ell)} - (L z)_{(\ell)} \leq 0.
    \end{align*}

    \item \textbf{Case 2}: Control signal is within the allowed operating region, $-\overline{u}_{(\ell)} \leq u_{(\ell)} \leq \overline{u}_{(\ell)}$. Here, $\psi(u_{(\ell)}) = 0$, since $\mathrm{sat}(u_{(\ell)}) = u_{(\ell)}$.

    \item \textbf{Case 3}: Control signal is below the minimum allowed limit, $u_{(\ell)} < -\overline{u}_{(\ell)}$. Then,
    \begin{align*}
        \psi(u_{(\ell)}) = u_{(\ell)} + \overline{u}_{(\ell)} = (K z)_{(\ell)} + \overline{u}_{(\ell)} < 0
    \end{align*}
    Thus, from \eqref{eq:saturation_interval}, it follows that
    \begin{align*}
        \overline{u}_{(\ell)} + (K z)_{(\ell)} - (L z)_{(\ell)} \geq 0.
    \end{align*}
\end{itemize}

From these three cases, it can be verified that the inequality in \eqref{eq:nl_sector} holds for all $z$ in \eqref{eq:Du}. This concludes the proof.
\end{proof}

The region $\mathcal{D}_u$ is defined as a polyhedral set in~\eqref{eq:Du} in the $z$-coordinates.
However, in the $\sigma$-coordinates, we can obtain an equivalent representation for this set as follows:
\begin{align}
     \label{eq:Du2}
    \mathcal{D}_{u} = \left\lbrace \sigma \in \mathbb{R}^{n} :
    |\left( K_{(\ell)} - L_{(\ell)} \right){\sigma}/{\|\sigma\|} | \leq \bar u_{\ell}, \;
    \ell \in \mathbb{N}_{\leq m}\right\rbrace.
\end{align}
Thus, in $\sigma$-coordinates, the half-planes defining the polyhedral set can be deformed, producing
surfaces in the $\sigma$-space.

\subsection{Control design condition}

The design condition for the UVC is derived in this section using the following Lyapunov function candidate:
\begin{align}\label{eq:Lyap-uvc-1}
  V(\sigma) = \frac{1}{\|\sigma\|} \sigma^\top P \sigma,
\end{align}
where $P = P^\top > 0$. This function can be rewritten as the following standard quadratic function using the $z$-coordinates:
\begin{align}\label{eq:Lyap-uvc-2}
  V(z) = z^\top P z.
\end{align}

\begin{theorem}\label{thm:2}
  Consider the uncertain system~\eqref{eq:plant} controlled by~\eqref{eq:control-uvc}. Given a scalar \(\mu > 0\), if there exist a symmetric matrix \( X \in \mathbb{R}^{n \times n} \), a diagonal symmetric matrix \( S \in \mathbb{R}^{m \times m} \), and full matrices \( Z, Y \in \mathbb{R}^{m \times n} \) such that the following conditions hold:
  \begin{align}
     & X > 0, \quad S > 0, \quad \Lambda_i < 0, \quad \forall i \in \mathbb{N}_{\leq{N}}, \label{eq:thm-1-2} \\
     & \begin{bmatrix}
         X                       & Z_{(\ell)}^{\top} - Y_{(\ell)}^{\top} \\
         Z_{(\ell)} - Y_{(\ell)} & \bar{u}_{(\ell)}^{2}
       \end{bmatrix} \geq 0, \quad \forall \ell \in \mathbb{N}_{\leq m}, \label{eq:thm-2-2}                  \\
  \end{align}
  where
  \begin{align}
     & \Lambda_i = \begin{bmatrix}
                     \Lambda_{i, 11} & Y^{\top} - B S & Z^\top B_i^\top \\
                     Y - S B^{\top}  & - 2S           & 0               \\
                     B_i Z           & 0              & -\mu I
                   \end{bmatrix},         \\
     & \Lambda_{i, 11} = B_i Z + Z^\top B_i^\top + \frac{\mu}{4} I + \tilde{Q}.
  \end{align}
  Then, the origin of the closed-loop system~\eqref{eq:closed-loop-uvc} is finite-time stable, where \(K = Z X^{-1}\) and \(L = Y X^{-1}\).  Moreover, the set
  \begin{equation}
    \label{eq:Dx}
    \Omega = \left\lbrace z \in \mathbb{R}^{n} : z^{\top} P z \leq 1 \right\rbrace,
  \end{equation}
  where $P = X^{-1}$, provides an estimate of the domain of attraction of the equilibrium of~\eqref{eq:closed-loop}, and it satisfies \(\Omega \subset \mathcal{D}_u\).
\end{theorem}
\begin{proof}
  Consider the quadratic Lyapunov candidate function defined in \eqref{eq:Lyap-uvc-2}. By multiplying both sides by \(\operatorname{diag}(X^{-1}, I)\) and applying Schur’s complement lemma, it follows from \eqref{eq:thm-2-2} that
  \begin{align}
     & z^{\top} P z \geq \frac{|(Z_{(\ell)} - Y_{(\ell)}) z|^{2}}{\bar{u}^{2}_{(\ell)}}, \label{eq:inc_Du_ineq} \\
     & V(z) \geq \frac{|(K_{(\ell)} - L_{(\ell)}) z|^{2}}{\bar{u}_{(\ell)}^{2}}, \nonumber
  \end{align}
  for all \(\ell \in \mathbb{N}_{\leq m}\), where \(P = X^{-1}\).
  Since \(z \in \mathcal{D}_{z}\), it follows from \eqref{eq:inc_Du_ineq} that \(\mathcal{D}_{z} \subset \mathcal{D}_{u}\).


  Now, assume that conditions \eqref{eq:thm-1-2}--\eqref{eq:thm-2-2} hold. From \eqref{eq:thm-1-2}, it follows that both \(X\) and \(S\) are nonsingular, and since \(X, S > 0\), their inverses \(X^{-1}\) and \(S^{-1}\) exist. Multiplying the inequalities in \eqref{eq:thm-2-2} by \(\operatorname{diag}(X^{-1}, S^{-1}, I)\) on the left and its transpose on the right yields
  \begin{align}\label{eq:proof-thm2-1}
     & \begin{bmatrix}
         \tilde \Lambda_{i, 11} & L^{\top} U - PB & K^\top B_i^\top \\
         U L - B^{\top} P       & -2 U            & 0               \\
         B_i K                  & 0               & -\mu I
       \end{bmatrix} < 0,                        \\
     & \tilde \Lambda_{i, 11} = PB_iK + K^\top B_i^\top P + \frac{\mu}{4} P^2 + Q.
  \end{align}
  for all \(i \in \mathbb{N}_{\leq{N}}\), where \(U = S^{-1},\) \(K = ZX^{-1}\), and \(L = Y X^{-1}\).
  Since \(B \in \operatorname{co}\{B_i\}_{i=1}^N\), multiplying \eqref{eq:proof-thm2-1} by \(\alpha_i\), summing from \(1\) to \(N\), and applying Schur complement, we obtain
  \begin{gather}
    \begin{bmatrix}
      \varPi           & L^{\top} U - PB \\
      U L - B^{\top} P & -2 U
    \end{bmatrix} < 0, \label{eq:proof-thm2-2} \\
    \varPi =  \frac{1}{\mu} K^\top B^\top B K + \frac{\mu}{4} P^2 + P B K + K^\top B^\top P + Q.
  \end{gather}
  Multiplying \eqref{eq:proof-thm2-2} by \(\begin{bmatrix} z^{\top} r(\sigma) & \psi^{\top}(u)\end{bmatrix}\) on the left and its transpose on the right leads to
  \begin{multline} \label{eq:proof-thm2-3}
    z^{\top} r(\sigma) \varPi r(\sigma) z
    - 2 \psi^{\top} U (\psi(u) - L r(\sigma)z) \\
    - \psi^{\top}(u) B^{\top} P r(\sigma) z
    - z^{\top} r(\sigma) P B \psi(u)
    < 0.
  \end{multline}
  Using Lemma~\ref{lemma:deadzone} and assuming \eqref{eq:thm-2-2} holds, \eqref{eq:nl_sector} is valid, and \eqref{eq:proof-thm2-3} implies
  \begin{multline} \label{eq:proof-thm2-4}
    z^{\top} r(\sigma) \varPi r(\sigma) z
    - \psi^{\top}(u) B^{\top} P r(\sigma) z \\
    - z^{\top} r(\sigma) P B \psi(u)
    < 0.
  \end{multline}
  Provided that
  \begin{equation}    \label{eq:proof-thm2-5}
    - \frac{1}{2}K^\top B^\top \Pi_\sigma P - \frac{1}{2}P \Pi_\sigma BK \leq \frac{1}{\mu} K^\top B^\top B K + \frac{\mu}{4} P^2
  \end{equation}
  since
  \begin{align*}
    \left(\frac{1}{\sqrt{\mu}}BK + \frac{\sqrt{\mu}}{2}\Pi_\sigma  P \right)^\top \left(\frac{1}{\sqrt{\mu}}BK + \frac{\sqrt{\mu}}{2}\Pi_\sigma P \right) \geq 0
  \end{align*}
  and $\|\Pi_\sigma\|=1$, then, from \eqref{eq:proof-thm2-3}, it follows that
  \begin{multline} \label{eq:proof-thm2-6}
    -\frac{1}{2} z^{\top} r(\sigma) K^\top B^\top \Pi_\sigma P r(\sigma) z
    -\frac{1}{2} z^{\top} r(\sigma) P \Pi_\sigma B K r(\sigma) z \\
    + z^{\top} r(\sigma) P B K r(\sigma) z
    + z^{\top} r(\sigma) K^\top B^\top P r(\sigma) z \\
    + 2\eta z^{\top} r(\sigma) P r(\sigma) z
    - \psi^{\top}(u) B^{\top} P r(\sigma) z \\
    - z^{\top} r(\sigma) P B \psi(u) < 0
  \end{multline}
  The inequality \eqref{eq:proof-thm2-6} implies
  \begin{multline} \label{eq:proof-thm2-7}
    \dot{V}(z) = \dot z^{\top} P z + z^{\top} P \dot z < - z^\top r(\sigma)Qr(\sigma)z < 0,
  \end{multline}
  with $V(z)$ defined in~\eqref{eq:Lyap-uvc-2}.
  To show the finite-time convergence, notice that
  $z^\top r(\sigma) Q r(\sigma) z \geq \lambda_{\min}(Q) {\|z\|^2}/{\|\sigma\|} =  \lambda_{\min}(Q)$,
  hence, it is possible to obtain from \eqref{eq:proof-thm2-7} that the reaching time is upper-bounded by
  \begin{align}
    T_r \leq \frac{V_0}{\lambda_{\min}(Q)},
  \end{align}
  where $V_0 = V(\sigma(0)) = \sigma^\top(0) P \sigma(0) / \|\sigma(0)\|$, $\forall \sigma(0) \neq 0$.  This shows that the convergence occurs in finite time.
  This concludes the proof.
\end{proof}

\subsection{Determining the set of guaranteed reaching time}

In this section, we formulate a convex optimization problem for designing the UVC.
For a given upper-bound for the reaching time, we obtain the largest set of initial
conditions for which the finite-time convergence is ensured in the presence of saturation.

Notice that for a given initial condition $\sigma(0)$ (associated to $V_0$), 
the reaching time can be minimized by increasing the smallest eigenvalue of $Q$. 
This objective can be achieved by incorporating the following constraint:
\begin{align}\label{eq:maximize_Q}
\begin{bmatrix}
    \tilde{Q} & X \\
    X & \rho I
\end{bmatrix} \geq 0.
\end{align}
From~\eqref{eq:maximize_Q}, it follows from Schur complement that
$\tilde{Q} - \rho^{-1} X^2 \geq 0$. By multiplying both sides by $X^{-1}$,
we have that $Q \geq \rho^{-1} I$, since $Q = X^{-1} \tilde{Q} X^{-1}$.
Thus, by minimizing $\rho$, the eigenvalues of $Q$ are maximized, thus
reducing the reaching time $T_{r}$.

This objective can be achieved by~\eqref{eq:maximize_Q}. However, notice that $X = P^{-1}$ in Theorem~\ref{thm:2}.
To enlarge the estimated set of initial conditions, we consider the following constraint:
\begin{align}\label{eq:maximize_U0}
  \begin{bmatrix}
    \varphi I & I \\ I & X
  \end{bmatrix} \geq 0.
\end{align}
The condition in~\eqref{eq:maximize_U0} implies from Schur complement that $\varphi I \geq X^{-1}$.
Thus, $P \leq \varphi I$ or still $V(z) \leq \varphi z^\top z$.
Hence, it is possible to conclude that $\mathcal{B}\subset \Omega$,
where $\mathcal{B} = \{z \in \mathbb{R}^n : z^\top z \leq \varphi^{-1}\}$ and
\begin{align}\label{eq:Omega-set-2}
  \Omega = \{\sigma \in \mathbb{R}^n : V(\sigma) \leq 1\}.
\end{align}
Thus, if $\varphi$ is minimized, the set $\Omega$ is enlarged. Therefore, if $\sigma(0)$ is taken inside of $\Omega$, the reaching time is guaranteed bounded by
\begin{align}
  T_r \leq \frac{V_0}{\lambda_{\min}(Q)} \leq {{\rho}}.
\end{align}

The optimization problem for obtaining the largest set of initial conditions associated with a guaranteed
reaching time $\rho$ is stated in the sequel:
\begin{flalign}
   & \min\limits_{X, S, Z, Y} \qquad \varphi \label{eq:optmization_problem}                                                     \\
   & \text{subject to}~\text{and LMIs in}~\eqref{eq:thm-1-2}, \eqref{eq:thm-2-2}, \eqref{eq:maximize_Q}, \eqref{eq:maximize_U0}. \nonumber \\
\end{flalign}

\section{Numerical Results}
\label{sec:sim}

\subsection{Example~1: Planar kinematic manipulator}
We consider a planar kinematic manipulator with an end-effector
image position coordinates $\sigma = [p_x,p_y]^\top \in \mathbb{R}^2$
given by an uncalibrated fixed camera with an optical orthogonal axis with respect to the
robot workspace plane~\cite{oliveira2014overcoming}. The uncertain model proposed by~\cite{geromel2024lmi} is described as
\begin{align}
  B(\phi) =
  \begin{bmatrix}
    \cos{(\phi)}  & \sin{(\phi)} \\
    -\sin{(\phi)} & \cos{(\phi)}
  \end{bmatrix}.
\end{align}
The matrix $B$ that depends on the uncertain rotation angle $\phi$
due to the uncalibrated camera. Let $\bar{\phi}$ be a given nominal angle,
the uncertainty can be modeled as the variation $\Delta \phi = \phi - \bar{\phi}$,
such that $|\Delta \phi| \leq \bar{\Delta}$ and $B(\phi) = B(\Delta \phi) B(\bar{\phi})$. With this description, it is possible to obtain
$N=4$ vertices for the polytopic representation of the uncertain matrix $B(\phi)$ associated with the variations of the following parameters:
\begin{align}
  \begin{bmatrix}
    \cos{(\Delta \phi)} \\
    \sin{(\Delta \phi)}
  \end{bmatrix}
  \in
  \mathrm{co}\left\lbrace
  \begin{bmatrix}
    \cos{(\bar{\Delta})} \\
    \sin{(\bar{\Delta})}
  \end{bmatrix},
  \begin{bmatrix}
    1 \\
    \sin{(\bar{\Delta})}
  \end{bmatrix}, \right. \nonumber \\
  \left.
  \begin{bmatrix}
    \cos{(\bar{\Delta})} \\
    -\sin{(\bar{\Delta})}
  \end{bmatrix},
  \begin{bmatrix}
    1 \\
    -\sin{(\bar{\Delta})}
  \end{bmatrix}
  \right\rbrace,
\end{align}
valid for all $|\Delta \phi| \leq \bar{\Delta}$, provided that $0 \leq \bar{\Delta} \leq \pi/2~\mathrm{rad}$.
For the conducted experiments, we assume that $\bar{\phi} = \pi/6$ and $\bar{\Delta} = \pi/4$.
For the conducted experiments, we assume that the saturation levels are $\overline{u}_1 = \overline{u}_2 = 2$.

The optimization problem~\eqref{eq:optmization_problem} is solved considering $\mu = 3$ and the upper-bound to the reaching time $T_r$ is taken as $\rho = 1$. The resulting robust control gain is 
\begin{align}
    K = \begin{bmatrix}
        -1.9368  &  1.1182 \\
        -1.1182  & -1.9368
    \end{bmatrix}.   
\end{align}
Notice that this robust control gain ensures the stabilization of the saturated closed-loop system 
for any convex combination of the $4$ vertices of the polytopic domain associated with the uncertain matrix $B$.

Moreover, the largest region of initial conditions $\Omega$ for the guaranteed reaching time $\rho = 1$~s is illustrated in Fig.~\ref{fig:region} together with the set $\mathcal{D}_u$ in~\eqref{eq:Du2}. 
It is clearly noticed that the obtained region of initial conditions $\Omega$ is contained in the 
set $\mathcal{D}_u$. As theoretically proven, this ensures the finite-time convergence of
closed-loop trajectories initiating in $\Omega$, even in the presence of saturating actuators. 
Finally, this figure also depicts several closed-loop trajectories 
with initial conditions taken inside the region $\Omega$, which converge within the guaranteed
reaching time.
\begin{figure}[!ht]
    \centering
    \includegraphics[width=\columnwidth]{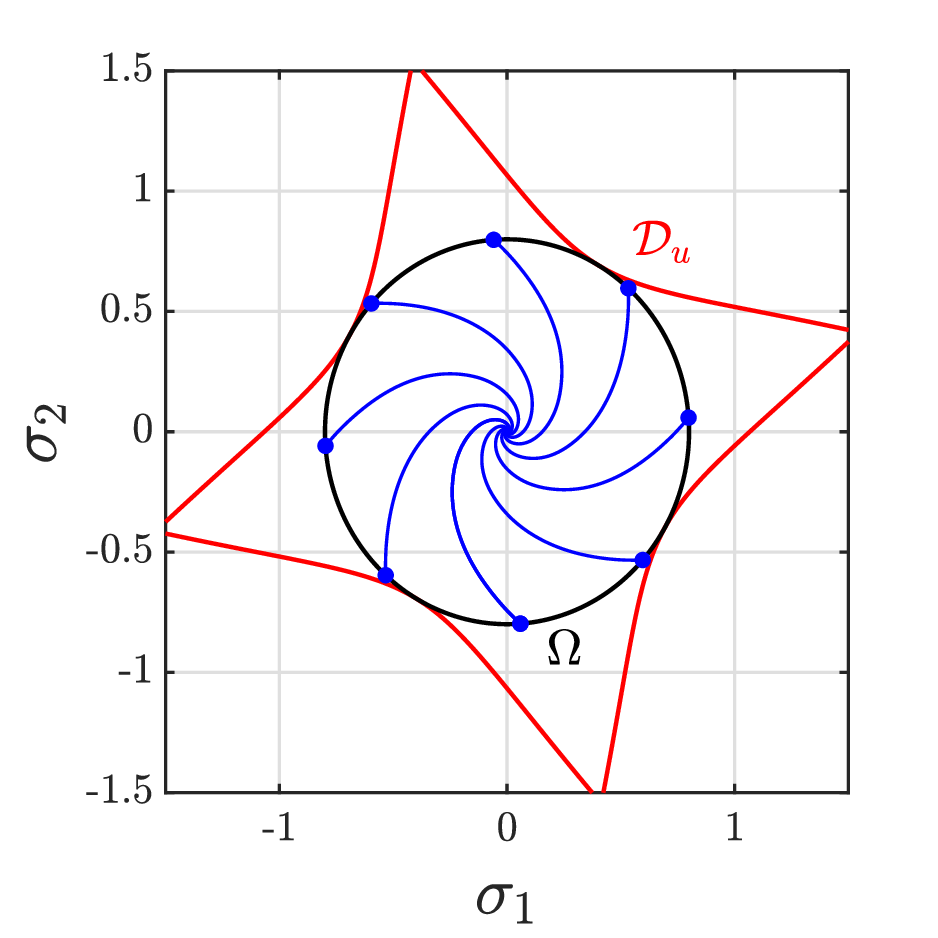}
    \caption{The set $\Omega$, defined in~\eqref{eq:Dx} for which the convergence occurs within the pre-specified reaching time $\rho$. The region $\Omega$ is contained in the region
    $\mathcal{D}_u$, given in~\eqref{eq:Du}.}
    \label{fig:region}
\end{figure}

The state trajectory with initial condition $\sigma(0) = [0.0587 \; -0.7976]^\top$, taken
at the border of $\Omega$, is depicted in Fig.~\ref{fig:state}(a). 
The saturated and unsaturated control input signals are depicted in Fig.~\ref{fig:state}(b).
The finite time convergence is ensured in a time smaller than $\rho = 1$~s, even though
the control input signals saturate. Even though the UVC can be viewed as a saturated-by-design
control law, one possible alternative to ensure the limitation of the control signal and avoid saturation 
is to constrain the norm of $K$. This is not done here since the control input signals
are allowed to saturate provided that the states are inside the region $\Omega \subset \mathcal{D}_u$. This fact is also evidenced by evaluating the norm of $K$, which is
$\|K\|_2 = 2.2364 > \overline{u} = 2$.
\begin{figure}[!ht]
    \centering
    \begin{subfigure}{\columnwidth}
    \centering
    \includegraphics[width=\linewidth]{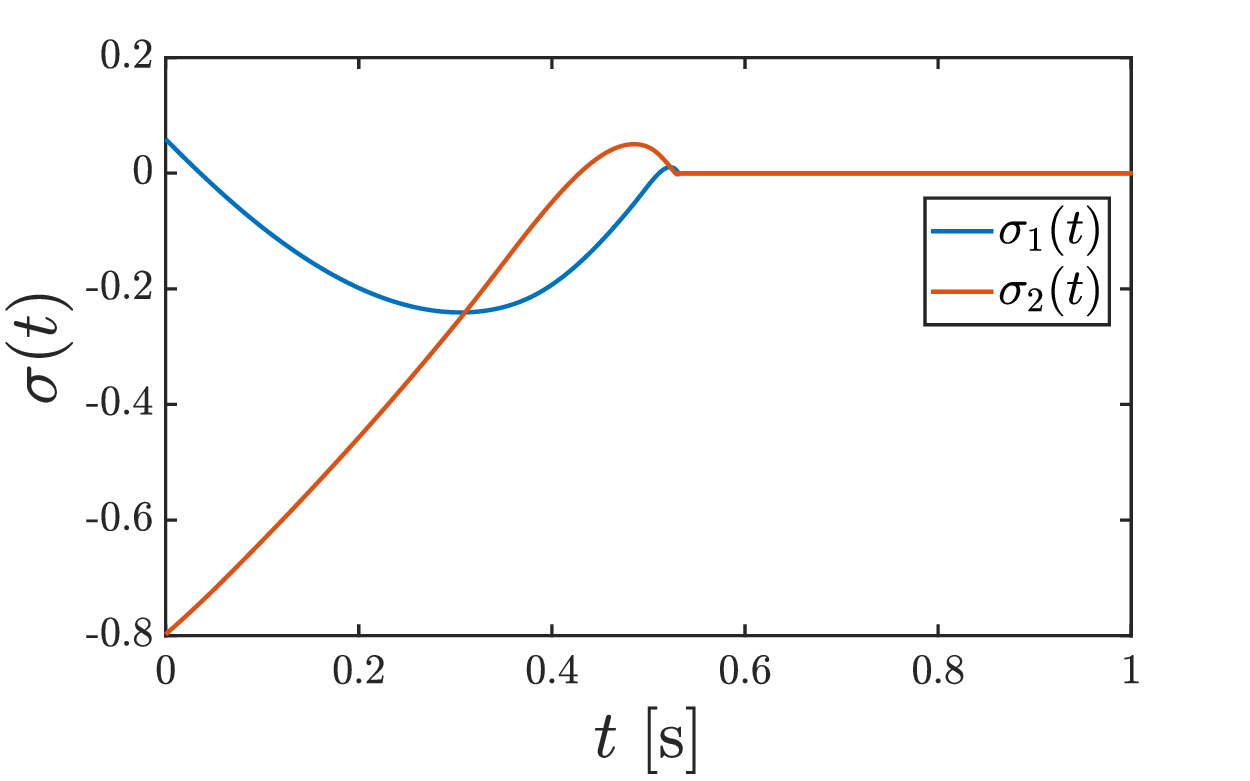}
    \caption{State $\sigma(t)$.}
    \end{subfigure}
    \begin{subfigure}{\columnwidth}
    \centering
    \includegraphics[width=\linewidth]{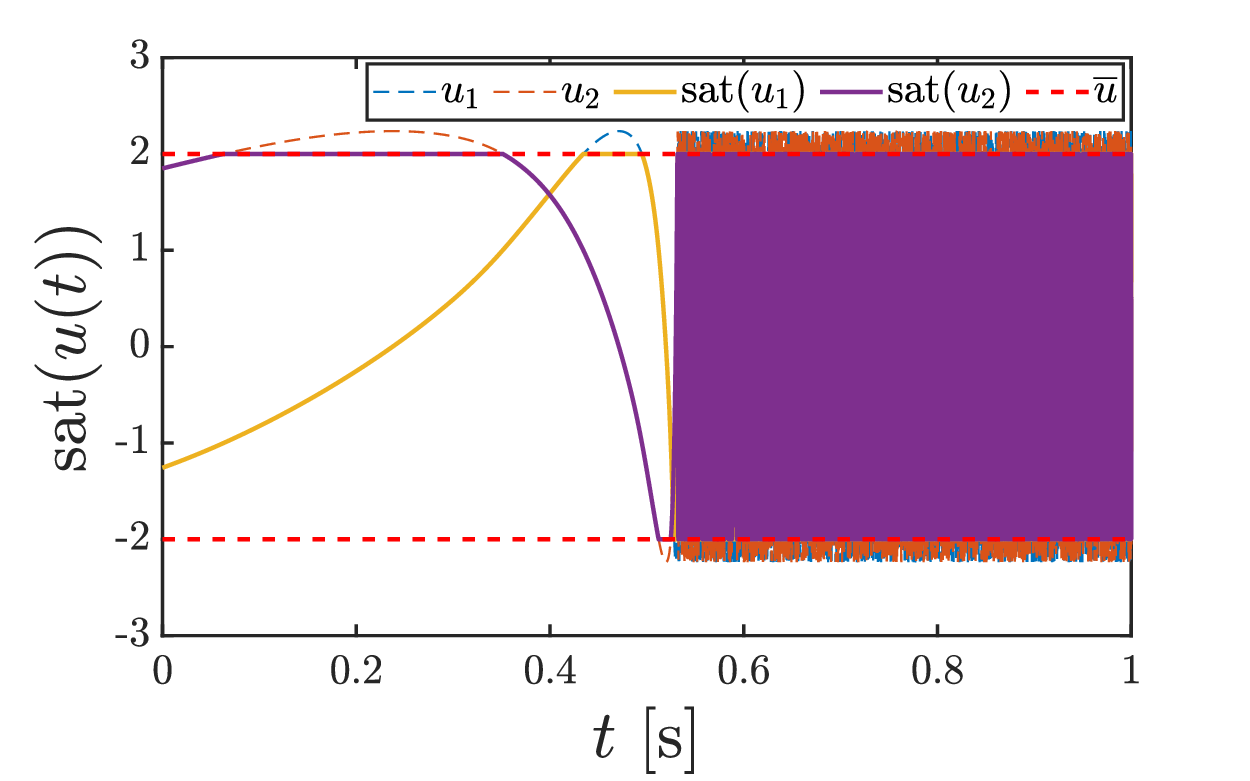}
    \caption{Control input $u(t)$ and the saturated input $\mathrm{sat}(u(t))$.}
    \end{subfigure}
    \caption{Closed-loop trajectory for the initial condition $\sigma(0) = [0.0587 \; -0.7976]^\top$ in $\Omega$.}
    \label{fig:state}
\end{figure}


\subsection{Example~2: Underwater ROV system}

The underwater remotely operated vehicle (ROV) dynamics is described with the state vector $\sigma=[v_x \; v_y \; \omega_z]^\top \in \mathbb{R}^3$, where $v_x$ and $v_y$ are velocities related to the body frame and $\omega_z$ is the angular velocity with respect to the $z$-axis. The system has four inputs resulting from propellers responsible for the displacement of the body. According to~\cite{geromel2024multivariable}, a simplified uncertain model of this system can be obtained by taking $B(g)=M^{-1}\Psi\Pi(g)$, with  $M=\mathrm{diag}(m_0, m_o, I_z)$, and
\begin{align}
    \Psi =
    \begin{bmatrix}
        \psi_1 & \psi_1 & \psi_1 &  \psi_1 \\
        \psi_1 & -\psi_1 & -\psi_1 & \psi_1\\
        -\psi_2 & \psi_2 & -\psi_2 & \psi_2
    \end{bmatrix},
\end{align}
where $m_0 = 290~\mathrm{kg}$ is the ROV mass, $I_z = 290~\mathrm{kg m}^2$ is the
moment of inertia, $\psi_1 = \sqrt{2}/2$, and $\psi_2 = 0.35~\mathrm{m}$. 
The input matrix with uncertain coefficients is $\Pi(g) = \mathrm{diag}(g_1,1,g_3,1)$,
where $g_1, g_3 \in [1/2,1]$ are uncertain gains in the actuator channels. 
Since the system has two uncertain parameters, it leads to $N=4$ vertices $B_i \in \mathbb{R}^{3 \times 4}$ in the polytopic description of the matrix $B$ in~\eqref{eq:plant}.
We assume the propellers are saturated by the limits $\overline{u} = 30$.

Then, the optimization problem in~\eqref{eq:optmization_problem} is solved considering
the guaranteed reaching time of $\rho = 10$~s, and $\mu = 0.4$. The parameter $\mu$ has been 
obtained via a scalar search such that a feasible solution was attained with an enlarged set $\Omega$.
The resulting robust control gain is:
\begin{align*}
    K = 
    \begin{bmatrix}
        -30.9190 &  -5.7321  &  5.9126 \\
      -20.2414  &  23.8253  &  -0.3787 \\
      -31.0926  &  0.9531  &  4.8242 \\
      -22.5835 &  -26.6822  & -14.9989
    \end{bmatrix}.
\end{align*}
Notice that the norm of this control gain is $\|K\|_2 = 53.7608 > \overline{u} = 30$, which does not avoid saturation. 

The resulting enlarged set of initial conditions $\Omega$ and the surfaces associated with the set $\mathcal{D}_u$ (given as in~\eqref{eq:Du2}) are depicted in Fig.~\ref{fig:region2}. Thus, trajectories
initiating in $\Omega$ will converge in finite time with a time smaller than or equal to $\rho = 10$~s. 
\begin{figure}[!ht]
    \centering
    \includegraphics[width=0.93\columnwidth]{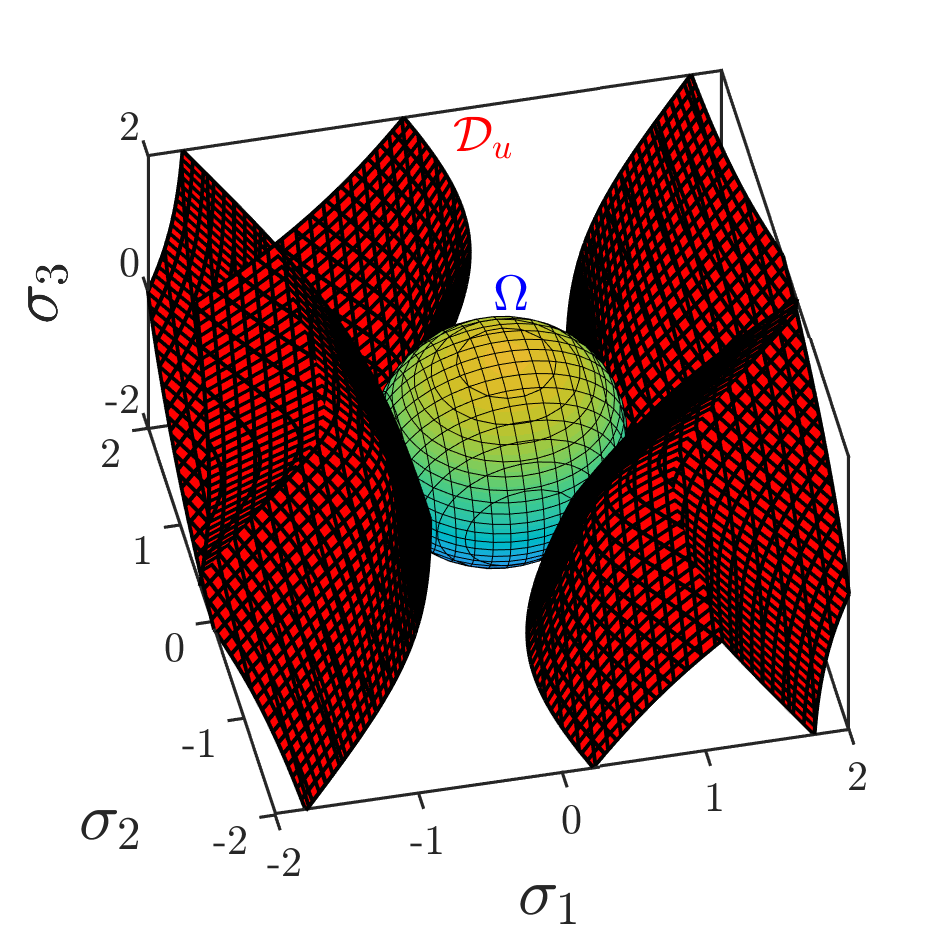}
    \caption{The set $\Omega$, defined in~\eqref{eq:Dx} for which the convergence occurs within the pre-specified reaching time $\rho$. The region $\Omega$ is contained in the region
    $\mathcal{D}_u$, given in~\eqref{eq:Du}.}
    \label{fig:region2}
\end{figure}

Consider the initial condition~$\sigma(0) = [0.60 \, 0.60 \,  0.4712]^\top \in \Omega$.
The states $\sigma(t)$ are depicted in Fig.~\ref{fig:state2}(a), where we can observe
the convergence within the specified reaching time. Moreover, the control input signals $u(t)$
and their saturated versions $\mathrm{sat}(u(t))$, which are effectively applied to the system,
are depicted in Fig.~\ref{fig:state2}(b). Notice that the control input signal $u_4$ remains saturated
during approximately $6.12$~s. However, even in the presence of saturation, the closed-loop trajectory
converges within the specified reaching time, as expected from the theoretical results in~Theorem~\ref{thm:2}. Finally, the time-series of the Lyapunov function $V(\sigma(t))$ in~\eqref{eq:Lyap-uvc-1} is also depicted in Fig.~\ref{fig:state2}(c). Notice that $V(\sigma(0)) \leq 1$, whcih illustrates
that the selected initial condition~$\sigma(0)$ is contained in the $\Omega$.
\begin{figure}[!ht]
    \centering
    \begin{subfigure}{\columnwidth}
    \centering
    \includegraphics[width=\linewidth]{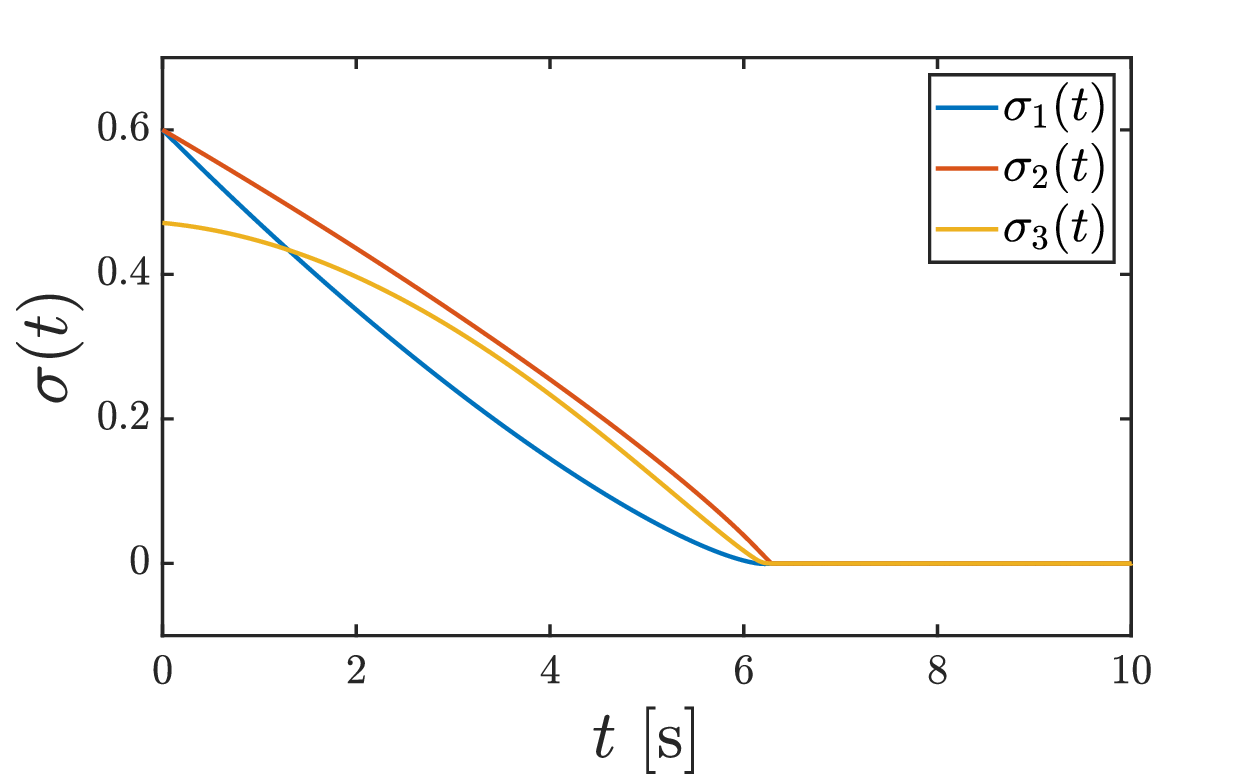}
    \caption{State $\sigma(t)$.}
    \end{subfigure}
    \begin{subfigure}{\columnwidth}
    \centering
    \includegraphics[width=\linewidth]{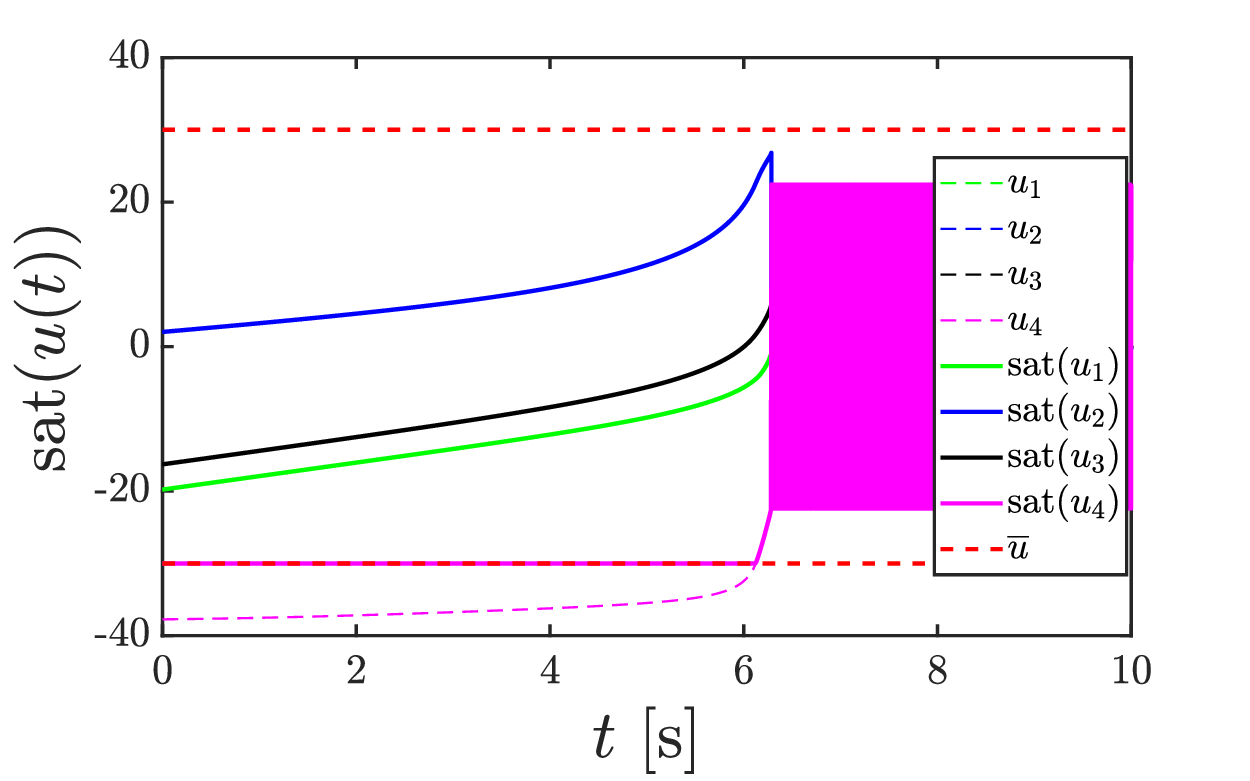}
    \caption{Control input $u(t)$ and the saturated input $\mathrm{sat}(u(t))$.}
    \end{subfigure}
    \begin{subfigure}{\columnwidth}
    \centering
    \includegraphics[width=\linewidth]{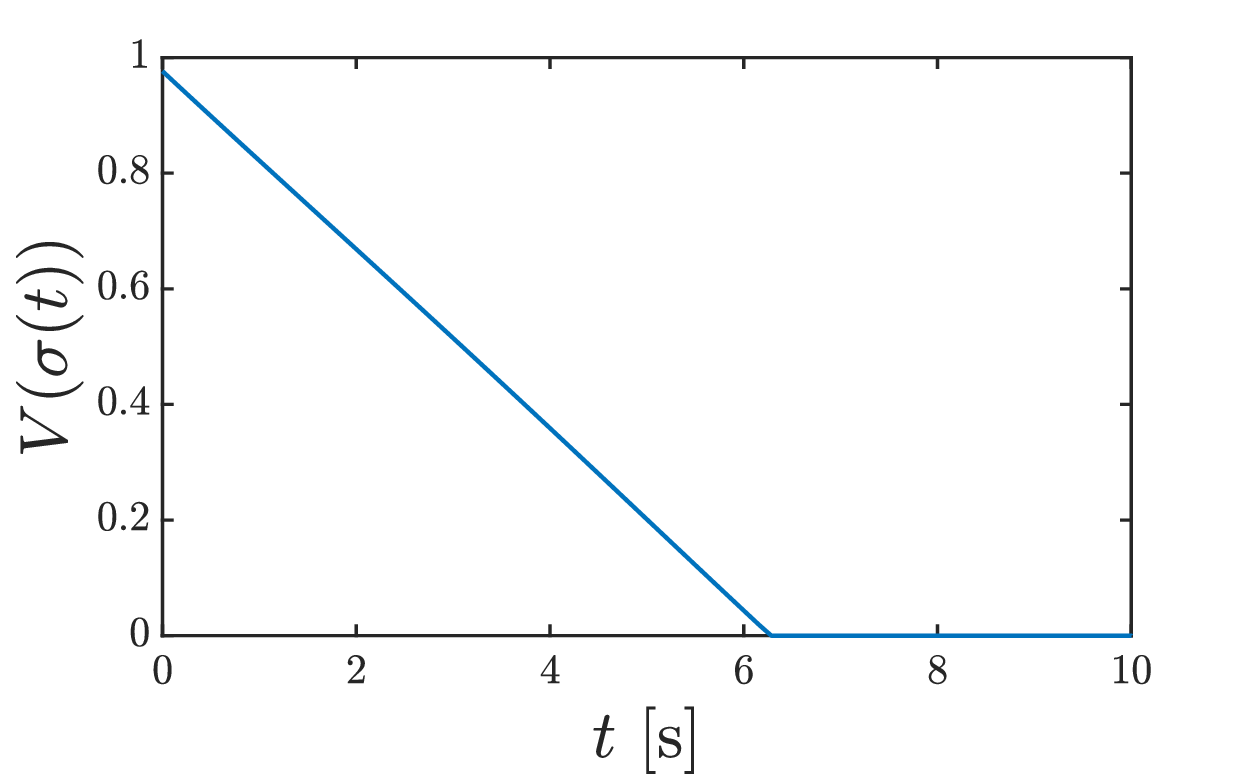}
    \caption{Lyapunov function $V(\sigma(t))$.}
    \end{subfigure}
    \caption{Closed-loop trajectory for the initial condition $\sigma(0) = [0.0587 \; -0.7976]^\top$ in $\Omega$.}
    \label{fig:state2}
\end{figure}

\section{Conclusion}
\label{sec:conclusion}

This paper has addressed the UVC design explicitly accounting for the presence
of saturating actuators. One of the main findings of this paper was a constructive
LMI-based condition to design the UVC gain for given saturation bounds. Then, 
we have provided a convex optimization problem to obtain an enlarged region of attraction estimation for which the convergence is satisfied for a 
pre-specified reaching time bound. Thus, we ensured that trajectories
initiating in this domain converge with a guaranteed reaching time, even
in the presence of saturating actuators. Interestingly, the proposed approach does not avoid
saturation by constraining the control gain norm under the saturation bound, which could be a way to deal with this issue using UVC laws. Future research will
investigate the extension of variable structure control design~\cite{coutinho2024systematic}, considering saturating actuators. Other possibilities lie in the design and analysis of different control problems with saturating actuators, as considered in the following references \cite{paper1,paper2,paper3,paper4,paper5,paper6,paper7,paper8,paper9,paper10,paper11,paper12,paper13,paper14,paper15,paper16,paper17,paper18,paper19,paper20}.


\bibliographystyle{IEEEtran}
\bibliography{references}

\begin{thebibliography}{10}
\providecommand{\url}[1]{#1}
\csname url@samestyle\endcsname
\providecommand{\newblock}{\relax}
\providecommand{\bibinfo}[2]{#2}
\providecommand{\BIBentrySTDinterwordspacing}{\spaceskip=0pt\relax}
\providecommand{\BIBentryALTinterwordstretchfactor}{4}
\providecommand{\BIBentryALTinterwordspacing}{\spaceskip=\fontdimen2\font plus
\BIBentryALTinterwordstretchfactor\fontdimen3\font minus
  \fontdimen4\font\relax}
\providecommand{\BIBforeignlanguage}[2]{{%
\expandafter\ifx\csname l@#1\endcsname\relax
\typeout{** WARNING: IEEEtran.bst: No hyphenation pattern has been}%
\typeout{** loaded for the language `#1'. Using the pattern for}%
\typeout{** the default language instead.}%
\else
\language=\csname l@#1\endcsname
\fi
#2}}
\providecommand{\BIBdecl}{\relax}
\BIBdecl

\bibitem{chen2024sliding}
W.-H. Chen, W.~Xu, and W.~X. Zheng, ``Sliding-mode-based impulsive control for
  a class of time-delay systems with input disturbance,'' \emph{Automatica},
  vol. 164, p. 111633, 2024.

\bibitem{Choi2007}
H.~H. Choi, ``{LMI-Based} sliding surface design for integral sliding mode
  control of mismatched uncertain systems,'' \emph{IEEE Transactions on
  Automatic Control}, vol.~52, no.~4, p. 736–742, Apr. 2007.

\bibitem{Cui2019}
Y.~Cui and L.~Xu, ``Chattering‐free adaptive sliding mode control for
  continuous‐time systems with time‐varying delay and process
  disturbance,'' \emph{International Journal of Robust and Nonlinear Control},
  vol.~29, no.~11, p. 3389–3404, Apr. 2019.

\bibitem{Wang2023}
B.~Wang, Y.~Shen, N.~Li, Y.~Zhang, and Z.~Gao, ``An adaptive sliding mode
  fault‐tolerant control of a quadrotor unmanned aerial vehicle with actuator
  faults and model uncertainties,'' \emph{International Journal of Robust and
  Nonlinear Control}, vol.~33, no.~17, p. 10182–10198, Feb. 2023.

\bibitem{Roy2020}
S.~Roy, S.~Baldi, and L.~M. Fridman, ``On adaptive sliding mode control without
  a priori bounded uncertainty,'' \emph{Automatica}, vol. 111, p. 108650, Jan.
  2020.

\bibitem{geromel2024lmi}
J.~Geromel, E.~Nunes, and L.~Hsu, ``{LMI}-based robust multivariable
  super-twisting algorithm design,'' \emph{IEEE Transactions on Automatic
  Control}, vol.~69, no.~7, pp. 4844 -- 4850, 2024.

\bibitem{geromel2024multivariable}
J.~C. Geromel, L.~Hsu, and E.~V.~L. Nunes, ``On multivariable super-twisting
  algorithm reaching time assessment,'' \emph{IEEE Transactions on Automatic
  Control}, vol.~69, no.~11, pp. 7972 -- 7979, 2024.

\bibitem{Corradini2007}
M.~Corradini and G.~Orlando, ``Linear unstable plants with saturating
  actuators: Robust stabilization by a time varying sliding surface,''
  \emph{Automatica}, vol.~43, no.~1, p. 88–94, Jan. 2007.

\bibitem{Han2020}
J.-S. Han, T.-I. Kim, T.-H. Oh, S.-H. Lee, and D.-I.~D. Cho, ``Effective
  disturbance compensation method under control saturation in discrete-time
  sliding mode control,'' \emph{IEEE Transactions on Industrial Electronics},
  vol.~67, no.~7, p. 5696–5707, Jul. 2020.

\bibitem{Zaafouri2017}
C.~Zaafouri, B.~Torchani, A.~Sellami, and G.~Garcia, ``Uncertain saturated
  discrete‐time sliding mode control for a wind turbine using a two‐mass
  model,'' \emph{Asian Journal of Control}, vol.~20, no.~2, p. 802–818, Jul.
  2017.

\bibitem{Tarbouriech2011}
S.~Tarbouriech, G.~Garcia, J.~M.~G. da~Silva~Jr, and I.~Queinnec,
  \emph{Stability and stabilization of linear systems with saturating
  actuators}.\hskip 1em plus 0.5em minus 0.4em\relax Springer Science \&
  Business Media, 2011.

\bibitem{oliveira2014overcoming}
T.~R. Oliveira, A.~C. Leite, A.~J. Peixoto, and L.~Hsu, ``Overcoming
  limitations of uncalibrated robotics visual servoing by means of sliding mode
  control and switching monitoring scheme,'' \emph{Asian Journal of Control},
  vol.~16, no.~3, pp. 752--764, 2014.

\bibitem{coutinho2024systematic}
P.~H.~S. Coutinho, I.~Bessa, V.~H.~P. Rodrigues, and T.~R. Oliveira, ``A
  systematic {LMI} approach to design multivariable sliding mode controllers,''
  \emph{Arxiv}, pp. 1--8, 2024, https://doi.org/10.48550/arXiv.2411.10592.

\bibitem{paper1}
T.~R. Oliveira, V.~H.~P. Rodrigues, and L.~Fridman, ``Generalized model
  reference adaptive control by means of global {HOSM} differentiators,''
  \emph{IEEE Transactions on Automatic Control}, vol.~64, no.~5, pp.
  2053--2060, 2018.

\bibitem{paper2}
D.~Rusiti, G.~Evangelisti, T.~R. Oliveira, M.~Gerdts, and M.~Krstic,
  ``Stochastic extremum seeking for dynamic maps with delays,'' \emph{IEEE
  Control Systems Letters}, vol.~3, no.~1, pp. 61--66, 2019.

\bibitem{paper3}
T.~R. Oliveira, A.~J. Peixoto, and E.~V.~L. Nunes, ``Binary robust adaptive
  control with monitoring functions for systems under unknown
  high‐frequency‐gain sign, parametric uncertainties and unmodeled
  dynamics,'' \emph{International Journal of Adaptive Control and Signal
  Processing}, vol.~30, no. 8-10, pp. 1184--1202, 2016.

\bibitem{paper4}
L.~L. Gomes, L.~Leal, T.~R. Oliveira, J.~P. V.~S. Cunha, and T.~C. Revoredo,
  ``Unmanned quadcopter control using a motion capture system,'' \emph{IEEE
  Latin America Transactions}, vol.~14, no.~8, pp. 3606--3613, 2016.

\bibitem{paper5}
C.~L. Coutinho, T.~R. Oliveira, and J.~P. V.~S. Cunha, ``Output-feedback
  sliding-mode control via cascade observers for global stabilisation of a
  class of nonlinear systems with output time delay,'' \emph{International
  Journal of Control}, vol.~87, no.~11, pp. 2327--2337, 2014.

\bibitem{paper6}
N.~O. Aminde, T.~R. Oliveira, and L.~Hsu, ``Global output-feedback extremum
  seeking control via monitoring functions,'' \emph{52nd IEEE Conference on
  Decision and Control}, pp. 1031--1036, 2013.

\bibitem{paper7}
A.~J. Peixoto, T.~R. Oliveira, L.~Hsu, F.~Lizarralde, and R.~R. Costa, ``Global
  tracking sliding mode control for a class of nonlinear systems via variable
  gain observer,'' \emph{International Journal of Robust and Nonlinear
  Control}, vol.~21, no.~2, pp. 177--196, 2011.

\bibitem{paper8}
V.~H.~P. Rodrigues and T.~R. Oliveira, ``Global adaptive {HOSM} differentiators
  via monitoring functions and hybrid state-norm observers for output
  feedback,'' \emph{International Journal of Control}, vol.~91, no.~9, pp.
  2060--2072, 2018.

\bibitem{paper9}
T.~R. Oliveira and M.~Krstic, ``Newton-based extremum seeking under actuator
  and sensor delays,'' \emph{IFAC-PapersOnLine}, vol.~48, no.~12, pp. 304--309,
  2015.

\bibitem{paper10}
T.~R. Oliveira, A.~C. Leite, A.~J. Peixoto, and L.~Hsu, ``Overcoming
  limitations of uncalibrated robotics visual servoing by means of sliding mode
  control and switching monitoring scheme,'' \emph{Asian Journal of Control},
  vol.~16, no.~3, pp. 752--764, 2014.

\bibitem{paper11}
T.~R. Oliveira, A.~J. Peixoto, and L.~Hsu, ``Peaking free output‐feedback
  exact tracking of uncertain nonlinear systems via dwell‐time and norm
  observers,'' \emph{International Journal of Robust and Nonlinear Control},
  vol.~23, no.~5, pp. 483--513, 2013.

\bibitem{paper12}
T.~R. Oliveira, J.~P. V.~S. Cunha, and L.~Hsu, ``Adaptive sliding mode control
  based on the extended equivalent control concept for disturbances with
  unknown bounds,'' \emph{Advances in Variable Structure Systems and Sliding
  Mode Control---Theory and Applications. Studies in Systems, Decision and
  Control}, vol. 115, pp. 149--163, 2017.

\bibitem{paper13}
H.~L. C.~P. Pinto, T.~R. Oliveira, and L.~Hsu, ``Sliding mode observer for
  fault reconstruction of time-delay and sampled-output systems---a time shift
  approach,'' \emph{Automatica}, vol. 106, pp. 390--400, 2019.

\bibitem{paper14}
T.~R. Oliveira, L.~R. Costa, J.~M.~Y. Catunda, A.~V. Pino, W.~Barbosa, and
  M.~N. de~Souza, ``Time-scaling based sliding mode control for neuromuscular
  electrical stimulation under uncertain relative degrees,'' \emph{Medical
  Engineering $\&$ Physics}, vol.~44, pp. 53--62, 2017.

\bibitem{paper15}
T.~R. Oliveira, A.~J. Peixoto, and L.~Hsu, ``Global tracking for a class of
  uncertain nonlinear systems with unknown sign-switching control direction by
  output feedback,'' \emph{International Journal of Control}, vol.~88, no.~9,
  pp. 1895--1910, 2015.

\bibitem{paper16}
T.~R. Oliveira, V.~H.~P. Rodrigues, M.~Krstic, and T.~Basar, ``Nash equilibrium
  seeking in quadratic noncooperative games under two delayed
  information-sharing schemes,'' \emph{Journal of Optimization Theory and
  Applications}, vol. 191, no.~2, pp. 700--735, 2021.

\bibitem{paper17}
T.~R. Oliveira, L.~Hsu, and A.~J. Peixoto, ``Output-feedback global tracking
  for unknown control direction plants with application to extremum-seeking
  control,'' \emph{Automatica}, vol.~47, no.~9, pp. 2029--2038, 2011.

\bibitem{paper18}
T.~R. Oliveira, A.~J. Peixoto, E.~V.~L. Nunes, and L.~Hsu, ``Control of
  uncertain nonlinear systems with arbitrary relative degree and unknown
  control direction using sliding modes,'' \emph{International Journal of
  Adaptive Control and Signal Processing}, vol.~21, no. 8-9, pp. 692--707,
  2007.

\bibitem{paper19}
D.~Tsubakino, T.~R. Oliveira, and M.~Krstic, ``Extremum seeking for distributed
  delays,'' \emph{Automatica}, vol. 153, no. 111044, pp. 1--14, 2023.

\bibitem{paper20}
A.~Battistel, T.~R. Oliveira, V.~H.~P. Rodrigues, and L.~Fridman,
  ``Multivariable binary adaptive control using higher-order sliding modes
  applied to inertially stabilized platforms,'' \emph{European Journal of
  Control}, vol.~63, pp. 28--39, 2022.

\end{thebibliography}


\end{document}